\documentclass[a4paper]{amsart}

\usepackage[english]{babel}
\usepackage[latin1]{inputenc}
\usepackage{enumerate} 
\usepackage{latexsym} 
\usepackage{amsthm} 
\usepackage{amssymb} 
\usepackage{amsmath} 
\usepackage{verbatim} 
\usepackage{yfonts}
\usepackage[all,arc,curve,color,frame]{xy}
\usepackage{etoolbox} 
\usepackage[pdftex]{graphicx}	
\usepackage{multirow} 
\usepackage{tikz}
		
\newtheorem{theorem}{Theorem}[section]
\newtheorem{lemma}[theorem]{Lemma}
\newtheorem{corollary}[theorem]{Corollary}
\newtheorem{proposition}[theorem]{Proposition}
\newtheorem{example}[theorem]{Example}

\theoremstyle{definition}
\newtheorem{definition}[theorem]{Definition}

\newtheorem*{assumption*}{Standing Assumption}
\theoremstyle{remark}
\newtheorem{remark}[theorem]{Remark}
\newtheorem{notation}[theorem]{Notation}

\def\relay#1#2{%
  \expandafter\def\csname #1\endcsname{#2}%
}

\def\makecal#1{%
	\relay{c#1}{\ensuremath{\mathcal{#1}}}%
}
\newcommand{\makebb}[1]{\relay{bb#1}{\ensuremath{\mathbb{#1}}}}

\forcsvlist{\makecal}{X,Y,K,N,R,F,Q,P,U,O,M,E}
\forcsvlist{\makebb}{R,N,C,Q,D,Z,F,T}

\newcommand{\makemathop}[1]{\expandafter\DeclareMathOperator\expandafter{\csname #1\endcsname}{#1}}

\forcsvlist{\makemathop}{id, coker, ev, End}

\def\newspan{\operatorname{span}}
\def\supp{\operatorname{supp}}

\newcommand{\N}{\mathbb{N}}

\newcommand{\Z}{\mathbb{Z}}
\newcommand{\R}{\mathbb{R}}

\newcommand{\BB}{\mathcal{B}}

\newcommand{\CC}{\mathcal{C}}

\newcommand{\OO}{\mathcal{O}}
\newcommand{\UU}{\mathcal{U}}

\numberwithin{equation}{section}

\title{$L_{2,\bbZ}\otimes L_{2,\bbZ}$ does not embed in 
$L_{2,\bbZ}$}

\author{Nathan Brownlowe}
\address[Nathan Brownlowe]{School of Mathematics and Applied Statistics, University of 
Wollongong, Australia}
\email{nathanb@uow.edu.au}

\author{Adam P W S{\o}rensen}
\address[Adam P W S{\o}rensen]{Department of Mathematics, University of Oslo, Norway, and School of Mathematics and Applied Statistics, University of Wollongong, Australia}
\email{apws@math.uio.no}


\date{}

\begin{document}

\begin{abstract}
For a commutative ring $R$ with unit we investigate the embedding of tensor product algebras into the Leavitt algebra $L_{2,R}$. We show that the tensor product $L_{2,\bbZ}\otimes L_{2,\bbZ}$ does not embed in 
$L_{2,\bbZ}$ (as a unital $*$-algebra).
We also prove a partial non-embedding result for the more general $L_{2,R}\otimes L_{2,R}$.
Our techniques rely on realising Thompson's group $V$ as a subgroup of the unitary group of $L_{2,R}$.
\end{abstract}

\maketitle

\section{Introduction}

Since the beginning of the study of Leavitt path algebras (\cite{AbramsPinoOriginalLeavitt,AraMorenoPardoKTheoryGraphAlgberas}), it has been know that there is a strong connection between Leavitt path algebras and graph $C^*$-algebras. 
In particular, some of the remarkable classification results for purely infinite $C^*$-algebras have been shown to have partial algebraic analogues for Leavitt path algebras (\cite{AbramsAnhPardoMatrixRings, ALPS, RuizTomfordeInfiniteGraphs}). 
In \cite{BrownloweSorensenOne} and this paper the authors continue the tradition by studying possible algebraic analogs of Kirchberg's celebrated embedding Theorem: All separable, exact $C^*$-algebras embed into $\OO_2$.   

The first problem to overcome is to make sense of what is meant by the algebraic analogue. 
We certainly want to replace $\OO_2$ with $L_{2,R}$, but then have to decide how to translate the conditions of being separable and exact. 
Since $L_{2,R}$ has a countable basis as an $R$-module, it is only possible to embed $R$-algebras which have a countable basis into it, so we choose that as our algebraic version of separable.
What to replace exactness with seems a more prickly question.
Exactness of $C^*$-algebras has a multitude of equivalent but very different looking definitions \cite[Theorem~IV.3.4.18]{BlackadarOpAlgBook}, so we are forced to make a choice.
One definition, where the name exact comes from, says that a $C^*$-algebra $A$ is exact if the maximal tensor product by $A$ preserves short exact sequences. 
The natural algebraic translation of that is that $A$ is flat as a module. 
So a naive translation of Kirchberg's Embedding Theorem for involutive algebras might be: All flat involutive $R$-algebras with a countable basis embed $*$-homomorphicly into $L_{2,R}$.
However, we will see that this statement is false, at least when $R = \Z$ .
We also note that in the case where $R$ happens to be a field $K$ this formulation seems very ambitious since \emph{all} algebras over a field are flat.
While we do not have a counterexample to the naive translation in this case, it seems unlikely that all algebras with a countable basis should embed into $L_{2,R}$.

The main result of \cite{BrownloweSorensenOne} is that every Leavitt path algebra $L_R(E)$ of a countable graph $E$ over a commutive ring $R$ with unit embeds into the Leavitt algebra $L_{2,R}$.
The condition of a countable graph ensures the existence of a countable basis, and all Leavitt path algebras are flat as $R$-modules. (A concrete basis is constructed in \cite{AlahmadiAlsulamiLeavittDimension} for row-finite $E$, and this result is generalised to all graphs and all rings of coefficients in 
\cite[Corollary 1.5.14]{AbramsAraSilesBook}.)
In the present paper we move slightly outside of the realm of Leavitt path algebras, looking instead at tensor products of Leavitt path algebras. 
Through the work of \cite{AraCortinasTensorProducts}, we know that for $K$ a field the tensor product $L_{2,K} \otimes L_{2,K}$ is not isomorphic to any Leavitt path algebra. 
A potential embedding of $L_{2,K}\otimes L_{2,K}$ into $L_{2,K}$ is thus a natural question. 
Indeed, this question is an open problem in the subject, and has attracted enough attention to be included on the Graph Algebra Problem Page,\footnote{\tiny{http://www.math.uh.edu/$\sim$tomforde/GraphAlgebraProblems/GraphAlgebraProblemPage.html}} which is a repository of open problems for graph $C^*$-algebras and Leavitt path algebras that was established after the workshop ``Graph Algebras: Bridges between graph C*-algebras and Leavitt path algebras'' held at the Banff International Research Station in April, 2013.
This problem is also discussed in \cite{AbramsTheFirstDecade}. 

In our main result we provide an answer to this problem in the case $R=\Z$: there is no embedding of $L_{2,\Z}\otimes L_{2,\Z}$ into $L_{2,\Z}$. More precisely, we prove that there is no such unital $*$-algebra embedding. We restrict our attention to unital embeddings, even when considering $*$-homomorphisms over a ring $R$, because it is known that for $K$ a field any homomorphism from a unital ring into $L_{2,K}$ can be ``twisted\footnote{Formally: Let $A$ be a unital $R$-algebra and let $\phi \colon A \to L_{2,R}$ be a ring homomorphism. Since $\phi(1)$ is an idempotent, there exist $x,y \in L_{2,R}$ such that $xy = \phi(1)$, $yx = 1$. Then $\psi(a) = y \phi(a) x$ defines a unital ring homomorphism from $A$ to $L_{2,R}$. If $\phi$ is injective so is $\psi$.}'' into a unital homomorphism because all nonzero idempotents in $L_{2,K}$ are equivalent (see \cite[Theorem 3.5]{AraMorenoPardoKTheoryGraphAlgberas}). We discuss the idea of twisting $*$-homomorphisms into $L_{2,\Z}$ in Section~\ref{sec: added in proof}. Since $L_{2,\Z}$ is flat (as a $\Z$-module) and has a countable basis, a consequence of our nonembedding result is that the naive algebraic analogue to Kirchberg's Embedding Theorem described earlier does not hold.

We prove our main result by studying the unitaries in $L_{2,\Z}$; in particular we use that Thompson's group $V$ sits naturally as a subgroup  $\cU_V$ of the unitary group of $L_{2,R}$ for any $R$.
In our second result we eliminate the existence of any unital $*$-algebra embedding of $L_{2,R} \otimes L_{2,R}$ into $L_{2,R}$ under which the tensor $u\otimes v$ of any two full spectrum unitaries $u,v\in L_{2,R}$ is in $\cU_V$. 
For a ring of characteristic $0$ we can conclude that any potential unital $*$-algebra embedding of $L_{2,R}\otimes L_{2,R}$ into $L_{2,R}$ must send the tensor of full spectrum unitaries to unitaries in $L_{2,R}$ with nontrivial coefficients. 
We hope these results and their proofs will inform future work on solving ``does $L_{2,R}\otimes L_{2,R}$ embed into $L_{2,R}$'' for rings other than $\bbZ$.

\section{Preliminaries on $L_{2,R}$ and Thompson's group $V$}\label{sec: prelim}

In this section we recall the definition of $L_{2,R}$, and we discuss a representation of 
$L_{2,R}$ by endomorphisms on the free $R$-module generated 
by the infinite paths in the graph underlying $L_{2,R}$. We then discuss the 
unitaries in $L_{2,R}$ and their relationship to 
Thompson's group $V$.   

\begin{assumption*}
Throughout this paper $R$ will always be a commutative ring with unit.
\end{assumption*}

\subsection{The algebra $L_{2,R}$ and a representation}\label{subsec: the repn 
of L2}

We will be focussed on the graph
\vspace{-0.3cm}

\begin{equation}\label{eq: 2 loop graph}
\begin{tikzpicture}[baseline=(current  bounding  box.center)]

\node[circle, draw=black,fill=black, inner sep=1pt] (u) at (0,0) {};
    
\draw[-latex,thick] (u) .. controls +(225:1.5cm) and +(135:1.5cm) .. (u);
\draw[-latex,thick] (u) .. controls +(315:1.5cm) and +(45:1.5cm) .. (u);
        
\node at (-1,0) {$a$};
\node at (1,0) {$b$};

\end{tikzpicture}
\end{equation}
\vspace{-0.3cm}

\noindent and its Leavitt path algebra $L_{2,R}$, which is the universal 
$R$-algebra generated by elements $a,a^*,b,b^*$ satisfying 

\[
a^*a=b^*b=1=aa^*+bb^*. 
\]
Note that the relations imply that $a^*b = 0 = b^*a$.
When $R$ is a field $K$, $L_{2,K}$ is the 
Leavitt
algebra of module type $(1, 2)$ (see \cite{LeavittOriginal}). For every word $\alpha=\alpha_1\dots\alpha_n$ in $a$ and $b$ we set $\alpha^*:=\alpha_n^*\dots\alpha_1^*$. Every element of $L_{2,R}$ is a finite $R$-linear combination of elements of the form $\alpha\beta^*$, where $\alpha$ and $\beta$ are words in $a$ and $b$. The map $\alpha\beta^*\mapsto \beta\alpha^*$ extends to an $R$-linear 
involution of $L_{2,R}$. We write $L_{2,R}\otimes L_{2,R}$ to mean the tensor 
product balanced over $R$. The tensor product $L_{2,R}\otimes L_{2,R}$ is an 
involutive $R$-algebra with $(x\otimes y)^*=x^*\otimes y^*$.

The infinite paths in the graph underlying $L_{2,R}$ is the set $\{a,b\}^\N$. 
We denote the set of finite paths by $\{a,b\}^*$, and by $|\alpha|$ the number 
of edges, or length, of a path $\alpha\in\{a,b\}^*$. For $\alpha\in \{a,b\}^*$ 
and 
$\xi\in\{a,b\}^\N$, $\alpha \xi\in\{a,b\}^\N$ denotes the obvious concatenation. 
For each $\alpha\in\{a,b\}^*$ we denote by $Z(\alpha)$ the cylinder set
\[
Z(\alpha)=\{\alpha \xi:\xi\in\{a,b\}^\N\}\subset \{a,b\}^\N.
\]
 
We denote by $M$ the uncountably generated free $R$-module generated by $\{a,b\}^\N$, and by 
$\End_R(M)$ the endomorphism ring of $M$. Consider the maps $T_a,T_b \colon M\to M$ 
given by
\[
T_a(\xi)=a\xi\quad\text{ and }\quad T_b(\xi)=b\xi
\]
for all $\xi\in \{a,b\}^\N$, and extended to $R$-linear combinations in the 
natural way. 
Also 
consider $T_a^*,T_b^* \colon M\to M$ given by
\[
T_a^*(\xi)=
\begin{cases}
\xi' & \text{if $\xi=a\xi'$ for some $\xi'\in\{a,b\}^\N$}\\
0 & \text{otherwise}
\end{cases}
\]
and
\[
T_b^*(\xi)=
\begin{cases}
\xi' & \text{if $\xi=b\xi'$ for some $\xi'\in\{a,b\}^\N$}\\
0 & \text{otherwise}
\end{cases}
\]
for all $\xi\in \{a,b\}^\N$, and extended to $R$-linear combinations in the 
natural way. Then $T_a,T_a^*,T_b,T_b^*\in\End_R(M)$, and it is straightforward 
to check that 
\[
T_a^*T_a=T_b^*T_b=1=T_aT_a^*+T_bT_b^*.
\]
The universal property of $L_{2,R}$ gives us a unital homomorphism 
$\pi_T \colon L_{2,R}\to\End_R(M)$ satisfying $\pi_T(a)=T_a$, $\pi_T(a^*) = T_a^*$, $\pi_T(b)=T_b$ and $\pi_T(b^*)=T_b^*$. It 
follows from the Cuntz-Krieger Uniqueness Theorem 
\cite[Theorem~6.5]{TomfordeLeavittOverRing} that $\pi_T$ is injective.

\subsection{Thompson's group $V$ and unitaries in $L_{2,R}$}\label{subsec: V 
and 
unitaries}

We are concerned with 
unitaries in $L_{2,R}$, which are elements $u$ satisfying $u^*u=uu^*=1$. There 
are 
two canonical subsets of the unitary group 
$\UU(L_{2,R})$ of $L_{2,R}$. We denote by $\cU_1$ the subset of 
unitaries that 
can be written without 
coefficients:
\[
\cU_1 = \left\{ u\in \cU(L_{2,R}) : u=\sum_{i=1}^n \alpha_i\beta_i^* 
\text{ for some distinct pairs } \alpha_i, \beta_i 
	\in \{a,b\}^* \right\}.
\]
The other subset, in fact a subgroup, is a homomorphic image of Thompson's group $V$. Thompson introduced $V$, as well as the groups $F$ and $T$, in unpublished notes in 1965. A account of his work can be found in \cite{CannonFloydParry}. Thompson's group $V$ can be realised as a collection of homeomorphisms of the Cantor set. Following \cite{NekrashevychCuntzPimsner} we think of elements of Thompson's 
group $V$ as tables 
\[
	\begin{pmatrix}
		\beta_1 & \beta_2 & \cdots & \beta_n \\
		\alpha_1 & \alpha_2 & \cdots & \alpha_n
	\end{pmatrix}
\]
where $\alpha_i, \beta_i\in \{a,b\}^*$ such that 
\[
	\{a,b\}^\N = \bigsqcup_{i=1}^n Z(\alpha_i) = \bigsqcup_{i=1}^n Z(\beta_i). 
\] 
Each table determines a homeomorphism of $\{a,b\}^\N$ which maps an infinite 
path of the form $\beta_i \xi$ to $\alpha_i \xi$. The set of homeomorphisms defined 
by tables is a group under composition. The calculation that gives the product 
of two tables as another table mirrors exactly the one which gives the product 
of two spanning elements $\sum_{i=1}^n\alpha_i\beta_i^*$ as another spanning 
element; that is, which shows that $\cU_1$ is a subgroup. So the map
\[
	\begin{pmatrix}
		\beta_1 & \beta_2 & \cdots & \beta_n \\
		\alpha_1 & \alpha_2 & \cdots & \alpha_n
	\end{pmatrix}\mapsto \sum_{i=1}^n\alpha_i\beta_i^*
\]     
is a group homomorphism of $V$ into $\cU(L_{2,R})$. We denote its image by 
$\UU_V$. By definition of $\cU_1$ and $\cU_V$ we have $\cU_V \subseteq \cU_1$. 
We now show that the reverse containment holds when $R$ has characteristic 
0. The case for fields is covered by \cite[Lemma 3.3]{PardoTheIsmomorphismProblem}.  

\begin{lemma}\label{lem: char 0 means UV=U1}
If $R$ has characteristic $0$, then $\cU_V = \cU_1$.
\end{lemma}

\begin{proof}
We need to show that if $u=\sum_{i=1}^n\alpha_i\beta_i^*$ is unitary, then 
\[
\{a,b\}^\N = \bigsqcup_{i=1}^n Z(\alpha_i) = \bigsqcup_{i=1}^n Z(\beta_i).
\] 
We work with the representation $\pi_T \colon L_{2,R}\to \End_R(M)$. We 
know that $\pi_T(u^*u)=\pi_T(uu^*)$ is the identity endomorphism of $M$. 
First suppose 
$\bigcup_{i=1}^nZ(\beta_i)\not= \{a,b\}^\N$. Then for all 
$\xi\in\{a,b\}^\N\setminus \big(\bigcup_{i=1}^nZ(\beta_i)\big)$ we have $T_{\beta_i}^*\xi=0$, and hence
 \[
\pi_T(u^*u)\xi=\pi_T(u^*)\Big(\sum_{i=1}^nT_{\alpha_i}
(T_{\beta_i}^*\xi)\Big)=0,
\]
which contradicts that $\pi_T(u^*u)$ is the identity endomorphism on $M$. So 
we have 
$\bigcup_{i=1}^nZ(\beta_i)= \{a,b\}^\N$. A similar argument using 
$\pi_T(uu^*)$ shows that $\bigcup_{i=1}^nZ(\alpha_i)= \{a,b\}^\N$.

We now claim that $\{Z(\beta_i):1\le i\le n\}$ are mutually disjoint. Let 
$m=\max\{|\beta_i|:1\le i\le n\}$, and for each $1\le i\le n$ we let $X_i$ 
denote the set of paths of length $m-|\beta_i|$. Since each $\sum_{\gamma\in 
X_i}\gamma\gamma^*=1$, we have
\[
u=\sum_{i=1}^n\alpha_i\beta_i^*=\sum_{i=1}^n\alpha_i\Big(\sum_{\gamma\in 
X_i}\gamma\gamma^*\Big)\beta_i^*=\sum_{i=1}^n\sum_{\gamma\in 
X_i}\alpha_i\gamma(\beta_i\gamma)^*.
\]
We relabel this sum and write $u=\sum_{i=1}^p\mu_i\nu_i^*$, where each 
$\nu_i$ has the same length. We will prove that each $\nu_i$ is 
distinct. Suppose not for contradiction. By relabelling we can assume without 
loss of generality that $\nu_1=\nu_2=\dots=\nu_j$ for some $j\ge 
2$, and 
$\nu_1\not=\nu_k$ for all $j+1\le k\le p$. Then for any 
$\xi\in\{a,b\}^\N$ 
we have
\begin{equation}\label{eq: distinct betas zero char}
\pi_T(u^*u)\nu_1\xi=\pi_T(u^*)\Big(\sum_{i=1}^j\mu_i\xi\Big)
=\sum_{k=1}^p\sum_{i=1}^j 
T_{\nu_k}T_{\mu_k}^*\mu_i\xi=
j\nu_1\xi + x,
\end{equation}
where $x\in M$ is a linear combination with positive integer coefficients of 
basis elements of the form 
$\nu_k\eta$ for some $1\le k\le p$ and $\eta\in\{a,b\}^\N$. Since the 
characteristic of $R$ is 0, it 
follows that the expression on the right of (\ref{eq: distinct betas zero 
char}) is not 
$\nu_1\xi$. But this contradicts that $\pi_T(u^*u)$ is the identity 
endomorphism. Hence each $\nu_i$ is distinct. 

Now, each $\nu_j$ has the form $\beta_i\gamma$. So if each $\nu_j$ is 
distinct, and because they are all of the same length, then it follows that no 
$\beta_i$ can be the subword of any other $\beta_k$ for $1\le i,k\le n$. Hence 
$\{Z(\beta_i):1\le i\le n\}$ are mutually disjoint, as claimed. A similar 
argument using 
$\pi_T(uu^*)$ shows that the $Z(\alpha_i)$ must be mutually disjoint.
\end{proof}

The following example shows that the assumption of characteristic $0$ is necessary. 

\begin{example} \label{ex:UV not U1}
Consider the field $K = \Z/2\Z$. 
In $M_4(\Z/2\Z)$ the matrix 
\[
\left(\begin{matrix}
0 & 1 & 1 & 1 \\
1 & 0 & 1 & 1 \\
1 & 1 & 0 & 1 \\
1 & 1 & 1 & 0 
\end{matrix} 
\right)
\]
is a self-adjoint unitary.
We can then use an embedding $M_4(\bbZ/2\Z)\hookrightarrow L_{2,\Z/2\Z}$ (see \cite[Example 5.3]{BrownloweSorensenOne}) to get a unitary $u \in L_{2,\Z/2\Z}$ given by
\[
 u = (ab+ba+bb)a^*a^*+(aa+ba+bb)b^*a^*+(aa+ab+bb)a^*b^*+(aa+ab+ba)b^*b^* .
\]
Clearly $u \in \cU_1$. 
We claim that $u \notin \cU_V$.
When we represent the elements of $L_{2,\Z/2\Z}$ as endomorphisms of the free module with basis $\{a,b\}^\N$, as discussed in Section \ref{subsec: the repn of L2}, the elements of $\cU_V$ will map basis elements to basis elements, but for any $\xi \in \{a,b\}^\N$ we see that
\[
  u aa\xi = ab\xi + ba\xi + bb\xi. 
\]
Hence $u \notin \cU_V$.

\end{example}

\section{Embedding the Laurent polynomials $L_R[w,w^{-1},z,z^{-1}]$}\label{sec: LR(w,z)}

We prove a positive and negative embedding result for the Laurent polynomials 
$L_R[w,w^{-1},z,z^{-1}]$ in two commuting variables over $R$. We first show that $L_R[w,w^{-1},z,z^{-1}]$ embeds 
into 
$L_{2,R}\otimes L_{2,R}$. We then show that there is no $*$-algebraic embedding of 
$L_R[w,w^{-1},z,z^{-1}]$ into $L_{2,R}$ mapping the unitaries $w$ and $z$ 
into the image $\cU_V\subset \cU(L_{2,R})$ of Thompson's group $V$. 

\begin{proposition}\label{prop: L2R embeds into tensor}
Let $R$ be a commutative ring with unit. Then $L_R[w,w^{-1},z,z^{-1}]$ embeds unitally 
into $L_{2,R} \otimes L_{2,R}$ as $*$-algebras.
\end{proposition}

\begin{remark}\label{rem: facts about flats!}
To prove Proposition~\ref{prop: L2R embeds into tensor} we use the following 
general observation about flat modules. 
If $\cE$ is a flat right $R$-module, and $\phi \colon \cM \to \cN$ is an embedding of left $R$-modules, then from the exact sequence $0\to \cM\hookrightarrow \cN$, we 
get 
the exact sequence 
\[
	0 \to \cE \otimes \cM \stackrel{\id \otimes \phi}{\to} \cE \otimes \cN.
\]
Hence $\cE \otimes \cM$ embeds into $\cE \otimes \cN$. Now if $\cE$ and $\cM$ 
are flat $R$-modules (so both left and right $R$-modules), and $\cE$ embeds 
into $\cM$, then $\cE\otimes \cE$ embeds into $\cE\otimes \cM$ and $\cM\otimes 
\cE$ embeds into $\cM\otimes \cM$. Using the flip isomorphism then gives an 
embedding 
\[
	\cE \otimes \cE \hookrightarrow \cE \otimes \cM 
	\stackrel{\cong}{\longrightarrow} 
	\cM \otimes \cE \hookrightarrow \cM \otimes \cM. 
\]
Composing with a flip isomorphism once more, we see that if $\psi \colon \cE \to \cM$ is an embedding, then $\psi \otimes \psi$ is an embedding of $\cE \otimes \cE$ into $\cM \otimes \cM$.
\end{remark}

\begin{proof}[Proof of Proposition~\ref{prop: L2R embeds into tensor}]
We have $L_R[w,w^{-1},z,z^{-1}]\cong L_R[w,w^{-1}] 
\otimes 
L_R[z,z^{-1}]$, and it follows from \cite[Theorem~4.1]{BrownloweSorensenOne} 
that $L_R[w,w^{-1}]$ embeds unitally into $L_{2,R}$ as $*$-algebras. 
Moreover, $L_R[w,w^{-1}]$ and $L_{2,R}$ have countable bases and hence are flat.
So the result follows from Remark~\ref{rem: facts about flats!} 
applied to $\cE=L_R[w,w^{-1}]$ and $\cM=L_{2,R}$.
\end{proof}  

We state our nonembedding result.

\begin{theorem}\label{thm: nonembedding of Laurent}
Let $R$ be a commutative ring with unit. There does not exist a $*$-algebraic embedding 
\[
	\phi \colon L_R[w,w^{-1},z,z^{-1}] \to L_{2,R}
\]
with $\phi(w), \phi(z) \in \cU_V$. 
\end{theorem}

The result is a consequence of the structure of Thompson's group $V$.
The following proposition extracts results about $V$ from \cite{BleakSalazarFreeProducts} and reformulates them in to our setting. 

\begin{proposition} \label{prop:commutingInV}
If $u,v \in \cU_V$ commute, then there exists a nonzero polynomial $q(w,z)\in 
R[w,z]$ such that $q(u,v) = 0$. 
\end{proposition}

In the proof we will adopted the notational conventions of \cite{BleakSalazarFreeProducts}.
However to fit with the notation of the rest of this paper, we will choose $\{a,b\}^\N$ as our model for the Cantor set rather than $\{0,1\}^\N$.

\begin{proof}[Proof of Proposition~\ref{prop:commutingInV}]
Fix commuting $u,v\in \cU_V$. 
Since $\cU_V$ is the homomorphic image of $V$ and since $V$ is simple, there exist commuting elements $g,h \in V$ that are mapped to $u,v$. 
By \cite[Proposition 10.1]{BrinHigherDimensional} there exist $m,n \in \bbN$ 
such that all finite orbits of $g$ have length at most $m$ and all finite 
orbits of $h$ have length at most $n$.
Let $k = 1 \cdot 2 \cdot \cdots \cdot m$ and $l = 1 \cdot 2 \cdot \cdots \cdot 
n$, then $g^k$ and $h^l$ have no nontrivial finite orbits.
If $q_0(w,z)\in R[w,z]$ is such that $q_0(u^k,v^l) = 0$, then 
$q(w,z)=q_0(w^k,z^l)$ satisfies $q(u,v) = 0$. 
So it suffices to prove the result under the assumption that $g$ and $h$ 
have no nontrivial finite orbits. From here on out we will assume that. 

We wish to use \cite[Lemmas 2.5 and 2.6]{BleakSalazarFreeProducts} to 
construct our polynomial. 
The notion of a flow graph and the components of support of an element of 
Thompson's group $V$ is introduced in \cite[Section 
2.4]{BleakSalazarFreeProducts}.
Denote the components of support of $g$ by $X_1, X_2, \ldots, X_m$ and the 
components of support of $h$ by $Y_1, Y_2, \ldots, Y_n$. 

From the description of components of support we see that the $X_i$ are 
disjoint, that each $X_i$ has the form $X_i = \bigsqcup_{j=1}^{m_i} 
Z(\alpha_{i,j})$ for $m_i\in\N$ and paths $\alpha_{i,j}\in\{a,b\}^*$, and that 
$g$ acts 
as the identity on $X_0 = \{a,b\}^\N \setminus \bigcup_{i=1}^m X_i$.
For $i = 1, 2, \ldots, n$ we define projections $r_i \in L_{2,R}$ by 
\[
	r_i = \sum_{j=1}^{m_i} \alpha_{i,j} \alpha_{i,j}^*.
\] 
By construction the $r_i$ are orthogonal, and since $g$ maps $X_i$ to itself, we have $r_i u r_i = u r_i$. 
We define a projection $r_0 \in L_{2,R}$ by
\[
	r_0 = 1 - \sum_{i=1}^m r_i.
\]
Then $r_0 u r_0 = u r_0 = r_0$ since $g$ acts as the identity outside of its components of support.
We have 
\[
	u = \sum_{i=0}^m r_i u r_i.
\]

Similarly we associate orthogonal projections $s_1, \ldots, s_n \in L_{2,R}$ to the components of support of $h$ and define 
\[
	s_0 = 1 - \sum_{i=1}^n s_i 
\]
Then 
\[
	v = \sum_{i=0}^n s_i v s_i.
\]

Since $g$ and $h$ commute \cite[Lemma 2.5(3)]{BleakSalazarFreeProducts} tells 
us that the components of support of $g,h$ are either disjoint or equal.
Thus for $j = 1, 2, \ldots, n$, $s_j$ either equals $r_i$ for some $i$, or 
$s_j$ is a subprojection of $r_0$.
Similarly $r_i$ is a subprojection of $s_0$ if it is not equal to one of the $s_j$. 

Let $I = \{ 1 \leq i \leq m : r_i = s_j \text{ for some } 1 \leq j \leq n \}$, that is $I$ are the indices of the common components of support. 
Let $l$ be the size of $I$.
We let $p_1, p_2, \ldots, p_l$ be the associated projections and we put $p_0 = r_0$. 
Then 
\[
    1 - \sum_{k=0}^l p_k = \sum_{i \notin I \cup \{0\}} r_i
\]
is a subprojection of $s_0$.
Let $p_{l+1} = 1 - \sum_{k=0}^l p_k$. 

We now have orthogonal projections $p_0, p_1, \ldots, p_{l+1}$ such that 
\begin{itemize}
	\item $\sum_{k=0}^{l+1} p_k = 1$, 
	\item for $k = 1, 2, \ldots, l$ we have $p_k = r_{i(k)} = s_{j(k)}$ for 
	some $1 \leq i(k) \leq m$, $1 \leq j(k) \leq n$, 
	\item $p_0 = r_0$,
	\item $p_{l+1} \leq s_0$,
\end{itemize}

Furthermore, since $p_{l+1}$ is a sum of projections onto components of support of $g$, we also have that $p_{l+1}up_{l+1} = u p_{l+1}$. 
As $p_{l+1}$ is a subprojection of $s_0$, and so correspondents to a subset where $h$ acts as the identity, we also have $p_{l+1} v p_{l+1} = v p_{l+1} = p_{l+1}$.
Since $p_0$ will be a (possibly empty) sum of projections onto components of support of $h$ and a (possibly zero) subprojection of $s_0$ we see that $p_0 v p_0 = v p_0$.
Therefore we also have 
\begin{itemize}
	\item $u = \sum_{k=0}^{l+1} p_k u p_k$, and 
	\item $v = \sum_{k=0}^{l+1} p_k v p_k$. 
\end{itemize}

For each $1 \leq k \leq l$ we can use \cite[Lemma 2.6]{BleakSalazarFreeProducts} to find $m_k, n_k 
\in \bbN$ such that
\[
	(p_k u p_k)^{m_k} = (p_k v p_k)^{n_k}. 
\]
We define a polynomial $q \in R[z,w]$  by 
\[
	q(w,z) = \left( wz - z \right) \left( \prod_{k=1}^{l} (w^{m_k} - 
	z^{n_k})   \right) \left( z - wz \right).
\]
Note that $q$ is nonzero since the coefficient of the only $w^{1+m_1 + m_2 + 
\cdots m_l + 1}z^2$ term is $-1$. 

For $1 \leq k \leq l$ we see that $q(p_k u p_k, p_k v p_k) = 0$ because one of the middle factors of $q$ will be zero. 
By the definition of $p_0$ we have $p_0 u p_0 = p_0$ so $p_0 u p_0 p_0 v p_0 = p_0 v p_0$ and therefore $q(p_0 u p_0,p_0 v p_0) = 0$ since the left most factor of $q$ will be zero. 
Using that $p_{l+1} v p_{l+1} = p_{l+1}$ and the right most factor of $q$ similar reasoning shows that $q(p_l u p_l, p_l v p_l) = 0$. 

Because the $p_k$ are orthogonal we see that 
\[
	q(u,v) = q( \sum_{k=0}^l  p_k u p_k, \sum_{k=0}^l p_k v p_k) = 
	\sum_{k=0}^l q(p_k u p_k, p_k v p_k) = 0. \qedhere
\]
\end{proof}

\begin{proof}[Proof of Theorem~\ref{thm: nonembedding of Laurent}]
Suppose $\phi \colon L_R[w,w^{-1},z,z^{-1}] \to L_{2,R}$ is a $*$-homomorphism with 
$\phi(w), \phi(z) \in \cU_V$. Then we know from 
Proposition~\ref{prop:commutingInV} that there is a polynomial $q(w,z)\in 
R[w,z]$ with $q(\phi(w),\phi(z))=0$. So $\phi(q(w,z))=q(\phi(w),\phi(z))=0$, 
and hence $\phi$ is not injective.
\end{proof}

\section{$L_{2,\bbZ}\otimes L_{2,\bbZ}$ does not embed in $L_{2,\Z}$}\label{sec: tensor for 
integers}

We now state our main result. 

\begin{theorem}\label{thm: L2Z nonembedding}
There is no unital $*$-algebraic embedding of $L_{2,\Z}\otimes L_{2,\Z}$ into $L_{2,\Z}$. 
\end{theorem}

To prove Theorem~\ref{thm: L2Z nonembedding} we need a number of results. We have decided to stay only in the generality of the integers throughout, although it is probable that some results can be generalised to cater for more general rings. The first step is to give a characterisation of $\UU(L_{2,\Z})$, which we do in Proposition~\ref{prop: char of unitaries in L2Z}.

\begin{definition}
An expression
\[
		\sum_{i=1}^n \lambda_i \alpha_i\beta_i^*
\]
is in \emph{reduced form} if each $\lambda_i \in \Z \setminus \{0\}$ and $\bigsqcup_{i=1}^n Z(\alpha_i)=\{a,b\}^\N=\bigsqcup_{i=1}^n Z(\beta_i)$.
\end{definition}

\begin{lemma} \label{lem: distinct betas}
Let $u \in \cU(L_{2,\Z})$ be given and suppose  
\[
	u = \sum_{i=1}^n \lambda_i \alpha_i \beta_i^*,
\] 
with the pairs $(\alpha_i, \beta_i)$ distinct and $\lambda_i \in \Z \setminus \{0\}$.
If all the $\beta_i$ have the same length then they are all distinct. 
\end{lemma}
\begin{proof}
We work with the representation $\pi_T \colon L_{2,\Z}\to \End_\Z(M)$ discussed in Section~\ref{subsec: the repn of L2}.

Suppose, for contradiction, that some of the $\beta_i$ are equal. 
By relabelling we can assume without loss of generality that $\beta_1=\beta_2=\dots=\beta_j$, and $\beta_1 \not=\beta_k$ for all $j+1\le k\le n$ for some $j \geq 2$.
Let $\xi$ be the aperiodic path $abaabbaaabbb\dots\in\{a,b\}^\N$.
Then we have
\begin{equation}\label{eq: distinct betas}
\pi_T(u^*u)\beta_1\xi=\pi_T(u^*)\Big(\sum_{i=1}^j\lambda_i\alpha_i\xi\Big)
=\sum_{k=1}^n\sum_{i=1}^j 
\lambda_i\lambda_kT_{\beta_k}T_{\alpha_k}^*\alpha_i\xi.
\end{equation}
Now, for every pair $1\le i\not= k\le j$, if $\alpha_i$ extends $\alpha_k$ we write $\gamma_{i,k}$ for the path with $\alpha_i=\alpha_k\gamma_{i,k}$, and if $\xi$ extends $\gamma_{i,k}$, we write $\xi_{i,k}$ for the infinite path with $\xi=\gamma_{i,k}\xi_{i,k}$.
Then the right-hand expression in (\ref{eq: distinct betas}) is given by
\begin{equation}\label{eq: distinct betas 2}
\Big(\sum_{i=1}^j\lambda_i^2\Big)\beta_1\xi + \sum_{\substack{1\le i\not= k\le 
j \\ \alpha_i=\alpha_k\gamma_{i,k}}}\lambda_i\lambda_k\beta_1\gamma_{i,k}\xi 
+  \sum_{\substack{1\le i\not= k\le 
j \\ \xi=\gamma_{i,k}\xi_{i,k}}}\lambda_i\lambda_k\beta_1\xi_{i,k} +  m,
\end{equation}
where $m\in M$ is a linear combination of basis elements of the form 
$\beta_k\eta$ for some $j+1\le k\le n$ and $\eta\in\{a,b\}^\N$. 
By the choice $\xi$ we have $\beta_1\xi\not=\beta_1\gamma_{i,k}\xi$ and $\beta_1\xi\not=\beta_1\xi_{i,k}$ for any $1\le i\not= k\le j$.
Since $\pi_T(uu^*)$ is the identity on $M$, we must therefore have that 
\[
	\sum_{\substack{1\le i\not= k\le 
j \\ \alpha_i=\alpha_k\gamma_{i,k}}}\lambda_i\lambda_k\beta_1\gamma_{i,k}\xi 
+  \sum_{\substack{1\le i\not= k\le 
j \\ \xi=\gamma_{i,k}\xi_{i,k}}}\lambda_i\lambda_k\beta_1\xi_{i,k} +  m = 0
\]	 
and 
\[
	\Big(\sum_{i=1}^j\lambda_i^2\Big)\beta_1\xi = \beta_1 \xi. 
\]
Hence $\sum_{i=1}^j\lambda_i^2 = 1$ which contradicts that $j > 1$.
So all the $\beta_i$ are distinct. 
\end{proof}

\begin{proposition} \label{prop: char of unitaries in L2Z}
Every unitary $u \in \cU(L_{2,\Z})$ can be written in reduced form. 
Furthermore, if 
\[
	u = \sum_{i=1}^n \lambda_i \alpha_i\beta_i^*,
\]
is in reduced form then $\lambda_i \in \{-1,1\}$ for all $i = 1,2,\ldots,n$.
\end{proposition}  
\begin{proof} 
Let $u\in\cU(L_{2,\Z})$ be given and write 
\[
	u = \sum_{i=1}^m \lambda_i' \alpha_i' \beta_i'^*.
\] 
As argued in the proof of Lemma~\ref{lem: char 0 means UV=U1} we can extend the paths $\alpha'_i,\beta'_i$ to write $u$ like
\begin{eqnarray} \label{eqn: reduced form of u}
	u=\sum_{i=1}^n \lambda_i \alpha_i \beta_i^*
\end{eqnarray}
where the $\beta_i$ all have equal length.
By simply collecting like terms and discarding terms where $\lambda_i = 0$, we may assume that the pairs $(\alpha_i, \beta_i)$ are distinct and that $\lambda_i \in \Z \setminus \{ 0 \}$. 
It now follows from Lemma~\ref{lem: distinct betas} that all the $\beta_i$ are distinct and hence they have disjoint cylinder sets.	

We claim that the $\alpha_i$ also have disjoint cylinder sets.  
To see this, consider the unitary 
\[
	u^* = \sum_{i=1}^n \lambda_i \beta_i \alpha_i^*.
\]
Similar to in the proof of Lemma~\ref{lem: char 0 means UV=U1} we let $m=\max\{|\alpha_i|:1\le i\le n\}$, and for each $1\le i\le n$ we let $X_i$ denote the set of paths of length $m-|\alpha_1|$.
Since each $\sum_{\gamma\in X_i}\gamma\gamma^*=1$, we have
\[
 u^*=\sum_{i=1}^n \lambda_i \beta_i\alpha_i^*=\sum_{i=1}^n \lambda_i \beta_i\Big(\sum_{\gamma\in 
X_i}\gamma\gamma^*\Big)\alpha_i^*=\sum_{i=1}^n\sum_{\gamma\in 
X_i}\lambda_i\beta_i\gamma(\alpha_i\gamma)^*.
\]
For any $\gamma, \gamma' \in \{a,b\}^*$ we have that $\beta_k \gamma = \beta_j \gamma'$ if and only if $k = j$ and $\gamma = \gamma'$, since the $\beta_i$ all have the same length. 
So it follows from Lemma~\ref{lem: distinct betas} that the $\alpha_i\gamma$ 
are distinct. 
Hence no $\alpha_k$ is an initial segment of any other $\alpha_j$, and the cylinder sets $Z(\alpha_i)$ are disjoint. 

It is easy to see, as in the proof of Lemma~\ref{lem: char 0 means UV=U1}, 
that the union of the cylinder sets of both the $\beta_i$ and the $\alpha_i$ 
must equal $\{a,b\}^\N$, so we have 
\[
	\bigsqcup_{i=1}^n Z(\alpha_i)=\{a,b\}^\N=\bigsqcup_{i=1}^n Z(\beta_i).
\]
Thus (\ref{eqn: reduced form of u}) is a reduced form of $u$. 

Finally, if  
\[
	u = \sum_{i=1}^n \lambda_i \alpha_i\beta_i^*,
\]
is a reduced form of $u$, then for any $\xi \in \{a,b\}^\N$ and any $1 \leq k \leq n$ we have 
\[
	\beta_k \xi = \pi_T(u^* u) \beta_k \xi = \lambda_k^2 \beta_k \xi,
\]
so $\lambda_k^2 = 1$. 
Therefore $\lambda_i \in \{-1,1\}$ for all $1 \leq i \leq n$. 
\end{proof}

\begin{notation}\label{notation: u+}
For a unitary $u$ with reduced form $u=\sum_{i=1}^k\lambda_i\alpha_i\beta_i^*\in \cU(L_{2,\bbZ})$ we 
denote 
\[
u_+=\sum_{i=1}^k \alpha_i\beta_i^*\in L_{2,\bbZ}.
\]
\end{notation}

\begin{proposition}\label{prop: properties of unitaries in L2Z}
Unitaries in $L_{2,\bbZ}$ have the following properties.
\begin{enumerate}
\item[(1)] If $u\in\cU(L_{2,\Z})$, then $u_+\in \cU_V$.
\item[(2)] If $u,v\in\cU(L_{2,\Z})$ commute, then $u_+$ and 
$v_+$ commute.
\item[(3)] If $u,v\in\cU(L_{2,\Z})$ commute, and $0\not=q(w,z)\in 
\bbZ[w,z]$ with 
$q(u_+,v_+)=0$, then there exists $0\not=\tilde{q}(w,z)\in 
\bbZ[w,z]$ with $\tilde{q}(u,v)=0$.
\end{enumerate} 
\end{proposition}

To prove Proposition~\ref{prop: properties of unitaries in L2Z} we use the following 
results 
about endomorphisms of free abelian groups, i.e. $\Z$-modules. For 
$f,g$ endomorphisms of a free abelian group 
$G$ we denote by $\supp(f)=\{m\in G:f(m)\not=0\}$, and by $[f,g]$ the 
commutator $fg-gf$.

\begin{lemma}\label{lem: - commute implies + commute}
Let $G$ be a free abelian group with basis $\BB$. Let $f,g,s,t\in\End(G)$ 
satisfy
\begin{enumerate}
\item[(i)] the image of a basis element is another basis element or 
zero; and
\item[(ii)] $G=\supp(f)\sqcup \supp(g)\sqcup\{0\}=\supp(s)\sqcup 
\supp(t)\sqcup\{0\}$. 
\end{enumerate}
Then
\[
[f-g,s-t]=0\Longrightarrow [f+g,s+t]=0.
\]
\end{lemma}

\begin{proof}
We have 
\begin{align*}
[f-g,s-t]=0 & \iff [f,s]-[f,t]-[g,s]+[g,t]=0\\
&\iff [f,s]+[g,t]=[f,t]+[g,s],
\end{align*}
and
\begin{align*}
[f+g,s+t]=0 & \iff [f,s]+[f,t]+[g,s]+[g,t]=0\\
&\iff [f,s]+[g,t]=-([f,t]+[g,s]).
\end{align*}
So we need to prove that
\begin{align} \label{eqn:commutators}
[f,s]+[g,t]=[f,t]+[g,s]\Longrightarrow [f,s]+[g,t] = 0.
\end{align}
From here on out we will assume the left hand side (\ref{eqn:commutators}). 

It suffices to check (\ref{eqn:commutators}) on $\BB$. 
To do this we partition $\BB$ into four sets: 
\begin{align*} 
\BB_{f,s} = \BB\cap \supp(f) \cap \supp(s),\quad &\BB_{f,t} = \BB\cap \supp(f) 
\cap 
\supp(t),\\
 \BB_{g,s} = \BB\cap \supp(g) \cap \supp(s),\quad &\BB_{g,t} = \BB\cap 
\supp(g) \cap \supp(t).
\end{align*}
For $\xi \in \BB_{f,s}$ we have that $[g,t]\xi = 0$ and 
\[      
	([f,t] + [g,s])\xi = ft\xi - tf\xi + gs\xi - sg\xi = -tf\xi + gs\xi. 
\]
Hence, on $\BB_{f,s}$ the left hand side of (\ref{eqn:commutators}) reduces to 
\begin{align} \label{eqn:SSAC1}
	fs - sf = -tf + gs.
\end{align}
Reordering the terms we get
\begin{align} \label{eqn:SSAC2}
	(f - g)s = (s - t)f.
\end{align}
Fix $\xi \in \BB_{f,s}$.
As $f,g,s,t$ all map basis elements to basis elements or zero, we cannot have $s\xi 
\in \supp(g)$ and $f\xi \in \supp(s)$, since then we would get the nonzero 
basis element $sf \xi$ equal to the negative of the nonzero basis element 
$gs\xi$. Similarly, we cannot have 
$s\xi \in 
\supp(f)$ and $f\xi \in \supp(t)$. Therefore, if $s\xi \in 
\supp(f)$ then we must have $f \xi \in \supp(s)$, and if 
$s\xi \in \supp(g)$ then we must have $f\xi \in \supp(t)$. In the first case, we see 
directly from (\ref{eqn:SSAC1}) that $[f,s]\xi = (fs 
- sf)\xi = 0$.
In the other case (\ref{eqn:SSAC2}) yields that $-gs \xi = -tf \xi$, which by 
(\ref{eqn:SSAC1}) gives that $[f,s]\xi = 0$.
As noted earlier $[g,t]\xi = 0$ for all $\xi \in \BB_{f,s}$, so we have that 
$[f,s] + [g,t] = 0$ on $\BB_{f,s}$. Hence (\ref{eqn:commutators}) is 
satisfied on $\BB_{f,s}$. We check the remaining three cases in a similar 
fashion:

For $\xi \in \BB_{f,t}$ we have that $[g,s]\xi = 0$ and 
\[
	([f,s] + [g,t])\xi = -sf\xi + gt\xi.
\]
So the left hand side of (\ref{eqn:commutators}) reduces to
\[
 -sf + gt = ft - tf, 
\] 
on $\BB_{f,t}$. 
Reordering we get 
\[
	(t - s)f = (f - g)t.
\]
Arguing as above, we note that for $\xi \in \BB_{f,t}$ we either have $f\xi 
\in \supp(t)$ and $t\xi \in \supp(f)$ or we have $f\xi \in \supp(s)$ and $t\xi 
\in \supp(g)$. 
Also as above, we see that in both those cases we get $[f,t]\xi = 0$.
Hence on $\BB_{f,t}$ we have
\[
	0 = [f,t] + [g,s] = [f,s] + [g,t].
\]
So (\ref{eqn:commutators}) is satisfied on $\BB_{f,t}$. 

On $\BB_{g,s}$ we have $[f,t] = 0$ and $[f,s] + [g,t] = fs - tg$, so that 
after reordering the terms, the left hand side of (\ref{eqn:commutators}) 
reads:
\[
	(f - g)s = (t - s)g.
\]
Similar to the above, we note that for $\xi \in \BB_{g,s}$ we either have 
$s\xi \in \supp(f)$ and $g\xi \in \supp(t)$, or we have $s\xi \in \supp(g)$ 
and 
$g\xi \in \supp(s)$.
In either case, we deduce that $[g,s] = 0$. 
Hence $[f,t] + [g,s] = 0$ on $\BB_{g,s}$, and so
(\ref{eqn:commutators}) holds on $\BB_{g,s}$. 

Finally, we see that $[f,s] = 0$ and $[f,t] + [g,s] = ft - sg$ on 
$\BB_{g,t}$.
By reordering the terms, we can write the left hand side of 
(\ref{eqn:commutators}) as
\[
	(s - t)g = (f - g)t.
\]
We conclude, as above, that for $\xi \in \BB_{g,s}$ we either have $g\xi \in 
\supp(s)$ and $t\xi \in \supp(f)$, or we have $g\xi \in \supp(t)$ and $t\xi 
\in 
\supp(g)$.
Either way, we get that $[g,t] = 0$ and therefore that $[f,s] + [g,t] = 0$ on 
$\BB_{g,t}$. Hence (\ref{eqn:commutators}) holds on $\BB_{g,t}$. 
\end{proof}

\begin{lemma}\label{lem:joint spectrum}
Let $G$ be a free abelian group with basis $\BB$. Let 
$f,g,s,t\in\End(G)$ satisfy
\begin{enumerate}
\item[(i)] the image of a basis vector is another basis vector or 
zero;
\item[(ii)] $G=\supp(f)\sqcup \supp(g)\sqcup\{0\}=\supp(s)\sqcup 
\supp(t)\sqcup\{0\}$; and 
\item[(iii)] $f + g, s + t$ and $f - g, s - t$ are pairs of commuting 
automorphisms. 
\end{enumerate} 
If $0\not=q(w,z) \in \Z[w,z]$ with $q(f+g,s+t) = 0$, then there exists 
$0\not=\tilde{q}(w,z)\in \Z[w,z]$ with $\tilde{q}(f-g,s-t) = 0$.
\end{lemma}

\begin{proof}
Fix nonzero $q(w,z) \in \Z[w,z]$ with $q(f+g,s+t) = 0$. Since multiplying 
$q(w,z)$ by $wz$ does not change that evaluation at 
$(f+g,s+t)$ is 
zero, we will assume for convenience that 
\[
q(w,z)= \sum_{i,j=1}^{n} k_{i,j} w^i z^j,
\]
for some $n\in\N$.


For every $1\le i,j\le n$ let $\Omega^{i,j}$ be the set of words in 
$\{f,g,s,t\}$ of length $i+j$ and with first $i$ letters in $\{f,g\}$ and last 
$j$ letters in $\{s,t\}$. Under composition, we view each $\Omega^{i,j}$ as a 
subset of $\End(G)$. Let $\CC$ be the collection of choice 
functions on 
$\{\Omega^{i,j}:1\le i,j\le n\}$. We claim that for every $\xi\in \BB$ there 
is a unique choice function $\varphi_\xi\in\CC$ such that 
$\varphi_\xi(\Omega^{i,j})\xi\not=0$ for all $1\le i,j\le n$. To see that 
there is a choice function $\varphi$ with this property, fix $1\le i,j\le n$. 
Let
\[
\omega_{i+j}^{i,j}\in\{s,t\}\text{ with }\omega_{i+j}^{i,j}\xi\not=0,
\] 
\[
\omega_l^{i,j}\in\{s,t\}\text{ with 
}\omega_l^{i,j}(\omega_{l+1}^{i,j}\dots\omega_{i+j}^{i,j}\xi)\not=0
\]
for each $i+1\le l\le i+j-1$, and 
\[
\omega_m^{i,j}\in\{f,g\}\text{ with 
}\omega_m^{i,j}(\omega_{m+1}^{i,j}\dots\omega_{i+j}^{i,j}\xi)\not=0
\] 
for each $1\le m\le i$. Then $\varphi\in\CC$ given by 
$\varphi(\Omega^{i,j})=\omega_1^{i,j}\dots 
\omega_{i+j}^{i,j}$ satisfies $\varphi(\Omega^{i,j})\xi\not=0$ for all $1\le 
i,j\le n$. We now note that, because the supports of $f$ and $g$ partition 
$G\setminus\{0\}$ and the supports of $s$ and $t$ partition $G\setminus\{0\}$, 
every choice of $\omega_1^{i,j},\dots,\omega_{i+1}^{i,j}$ is unique. Hence 
$\varphi_\xi=\varphi$ is the unique choice function with the desired 
property. 

For each 
$\varphi\in\CC$ let
\[
\BB_\varphi=\{\xi\in\BB:\varphi(\Omega^{i,j})\xi\not=0\text{ for all }1\le 
i,j\le n\}.
\]
So $\BB_\varphi=\{\xi\in\BB:\varphi_\xi=\varphi\}$. By the uniqueness of each 
$\varphi_\xi$ we have 
\[
\BB=\bigsqcup_{\varphi\in\CC} \BB_\varphi.
\]
Fix $\varphi\in\CC$. Then on $\BB_\varphi$ we have  
\[
q(f+g,s+t)= \sum_{i,j=1}^nk_{i,j}(f+g)^i(s+t)^j=
\sum_{i,j=1}^nk_{i,j}\varphi(\Omega^{i,j}),
\] 
and
\[
q(f-g,s-t)= 
\sum_{i,j=1}^nk_{i,j}(f-g)^i(s-t)^j=
\sum_{i,j=1}^n(-1)^{\alpha_{(\varphi,i,j)}}k_{i,j}\varphi(\Omega^{i,j}),
\]
where $\alpha_{(\varphi,i,j)}$ is the sum of the number of instances of $g$ 
and $t$ appearing in the 
word $\varphi(\Omega^{i,j})$. Let
\[
q_\varphi(w,z) =\sum_{i,j=1}^n(-1)^{\alpha_{(\varphi,i,j)}}k_{i,j}w^iz^j\in 
\Z[w,z].
\]
Then on $\BB_\varphi$ we have
\begin{align*}
q_\varphi(f-g,s-t) &= 
\sum_{i,j=1}^n(-1)^{2\alpha_{(\varphi,i,j)}}k_{i,j}\varphi(\Omega^{i,j})\\
&= 
\sum_{i,j=1}^n k_{i,j}\varphi(\Omega^{i,j})\\
&=q(f+g,s+t)\\
&=0.
\end{align*}
Define 
\[
\tilde{q}(w,z)=\prod_{\varphi\in\CC}q_\varphi(w,z)\in 
\Z[w,w^{-1},z,z^{-1}].
\]
Note that since each $\Omega^{i,j}$ is finite the set $\CC$ is finite, so the product is finite. 
For each $\xi\in \BB$ we use that polynomials commute to get
\[
\tilde{q}(f-g,s-t)\xi=\Big(\prod_{\varphi\not=\varphi_\xi}q_\varphi(f-g,s-t)\Big)
q_{\varphi_\xi}(f-g,s-t)\xi=0.
\]
Hence $\tilde{q}(f-g,s-t)=0$. Finally, to see that $\tilde{q}(w,z)$ is not the 
zero polynomial, let
\[
i_0=\max\{i: k_{i,j}\not=0\}\quad\text{and}\quad 
j_0=\max\{j: k_{i_0,j}\not=0\}.
\]
Then $\tilde{q}(z,w)$ contains the nonzero term
\[
(-1)^{\sum_{\varphi\in\CC}\alpha_{(\varphi,i_0,j_0)}}k_{i_0,j_0}^{|\CC|}
(w^{i_0}z^{j_0})^{|\CC|}. \qedhere
\]
\end{proof}

\begin{proof}[Proof of Proposition~\ref{prop: properties of unitaries in L2Z}]
Property (1) follows immediately from the characterisation of $\UU(L_{2,\Z})$ 
in Proposition~\ref{prop: char of unitaries in L2Z}. 

For (2) and (3) we let 
$u=\sum_{i=1}^m\lambda_i\alpha_i\beta_i^*$ and 
$v=\sum_{j=1}^n\mu_j\gamma_j\delta_j^*$ be commuting unitaries. Let
\begin{align*}
u_{+,+}=\sum_{\substack{1\le i\le m \\ \lambda_i=1}}\alpha_i\beta_i^*,\qquad 
& u_{+,-}=\sum_{\substack{1\le i\le m \\ \lambda_i=-1}}\alpha_i\beta_i^*,\\  
v_{+,+}=\sum_{\substack{1\le j\le n \\ \mu_j=1}}\gamma_j\delta_j^*,\qquad & 
v_{+,-}=\sum_{\substack{1\le j\le n \\ \mu_j=-1}}\gamma_j\delta_j^*.
\end{align*}
So we have 
\begin{align*}
u=u_{+,+}-u_{+,-},\qquad & u_+=u_{+,+}+u_{+,-},\\
v=v_{+,+}-v_{+,-},\qquad & v_+=v_{+,+}+v_{+,-}.
\end{align*}
Consider the $\Z$-module, or free abelain group, $M$ from Section~\ref{subsec: 
the repn of L2}, 
generated 
by basis $\BB=\{a,b\}^\N$, and the faithful representation $\pi_T \colon L_{2,\Z}\to 
\End(M)$. Let 
\[
f=\pi_T(u_{+,+}),\, g=\pi_T(u_{+,-}),\, s=\pi_T(v_{+,+})\text{ and } 
t=\pi_T(v_{+,-})
\]
For (2) first note that $f,g,s,t$ satisfy the hypotheses (i), (ii) of 
Lemma~\ref{lem: - commute implies + 
commute}. Moreover, we have 
\[
[f-g,s-t]=[\pi_T(u),\pi_T(v)]=\pi_T([u,v])=0.
\]
Lemma~\ref{lem: - commute implies + commute} now implies that
\[
\pi_T([u_+,v_+])=[\pi_T(u_+),\pi_T(v_+)]=[f+g,s+t]=0.
\]
Hence $[u_+,v_+]=0$, and (2) holds. Property (3) holds because (2) gives 
hypothesis (iii) from Lemma~\ref{lem:joint spectrum}, and then 
Lemma~\ref{lem:joint spectrum} gives the result.  
\end{proof}

\begin{proof}[Proof of Theorem~\ref{thm: L2Z nonembedding}]
We know from Proposition~\ref{prop: L2R embeds into tensor} that 
$L_\Z[w,w^{-1},z,z^{-1}]$ embeds 
into 
$L_{2,\Z}\otimes L_{2,\Z}$ by a unital $*$-homomorphism.
We claim that there is \emph{no} unital $*$-homomorphic embedding of $L_\Z[w,w^{-1},z,z^{-1}]$ into $L_{2,\Z}$. 
The result will then follow from this claim 
because any such embedding $L_{2,\Z}\otimes L_{2,\Z}\hookrightarrow L_{2,\Z}$ would 
induce an embedding $L_\Z[w,w^{-1},z,z^{-1}]\hookrightarrow L_{2,\Z}$, 
contradicting the claim. 

To see why the claim holds, let $\phi \colon L_\Z[w,w^{-1},z,z^{-1}]\to L_{2,\Z}$ be 
a unital $*$-homomorphism. 
Then $u=\phi(w)$ and $v=\phi(z)$ are commuting unitaries in 
$L_{2,\Z}$. We know from parts (1) and (2) of Proposition~\ref{prop: 
properties of unitaries in L2Z} that $u_+$ and $v_+$ are then commuting 
unitaries. Since 
$u_+,v_+\in \UU_V$, we know from Proposition~\ref{prop:commutingInV} that 
there is a nonzero polynomial $q(w,z)\in \Z[w,z]$ with $q(u_+,v_+)=0$. Now (3) 
of Proposition~\ref{prop: properties of unitaries in L2Z} says there is a 
nonzero polynomial 
$\tilde{q}(w,z)\in L_\Z[w,w^{-1},z,z^{-1}]$ with $\tilde{q}(u,v)=0$. So 
$\phi(\tilde{q}(w,z))=\tilde{q}(u,v)=0$, and hence $\phi$ is not injective. 
This proves the 
claim.
\end{proof}

\section{The case $L_{2,R}\otimes L_{2,R}$}\label{sec: general tensor 
case}

In this section we make some remarks about what our results indicate about the general question of whether $L_{2,R} \otimes L_{2,R}$ embeds into $L_{2,R}$. 
While we cannot answer the question completely we can give some partial results. 

Our result for $R=\Z$ depends on Proposition~\ref{prop: char of unitaries in L2Z}, which gives us  a nice characterisation of the unitaries in $L_{2,\Z}$. 
Since the difference in the characterisations of a general unitary and a unitary in $\cU_V$ is so minor (the allowance for minus signs in the former), Proposition~\ref{prop: properties of unitaries in L2Z} allows us to bump up the non-embedding result Theorem~\ref{thm: nonembedding of Laurent} which uses $\cU_V$ to the argument required to get a non-embedding $L_{2,\bbZ}\otimes L_{2,\bbZ}\not\hookrightarrow L_{2,\bbZ}$. 
In the general setting we lose such tight control on the unitaries as the following example illustrates.

\begin{example}
Consider the field $K = \R$.
There is a unital embedding of $M_2(\R)$ into $L_{2,\R}$ such that 
\[
  \begin{pmatrix}
      \lambda_{11} & \lambda_{12} \\
      \lambda_{21} & \lambda_{22}
  \end{pmatrix}
  \mapsto
  \lambda_{11}aa^* + \lambda_{12} ab^* + \lambda_{21} ba^* + \lambda_{22} bb^*.   
\]
So for any angle $\theta$ we can map the rotation unitary for $\theta$ into $L_{2,\R}$ to get a unitary 
\[
(\cos\theta)aa^*-(\sin\theta)ab^*+(\sin\theta)ba^*+(\cos\theta)bb^* \in L_{2,\R}.
\]
We immediately see that the coefficients of unitaries in $L_{2,\R}$ can take uncountably many values.
Furthermore, we cannot simply throw away the coefficients, like in Proposition \ref{prop: properties of unitaries in L2Z}, since $aa^* + ba^* + ab^* + bb^*$ is not a unitary in $L_{2,\R}$.
\end{example}

While we cannot see how to generalise our techniques beyond the case $R = \Z$, we can still apply our results to prove results about more general rings. 
Namely, we can show that if $K$ is a field of characteristic $0$, then any potential embedding of $L_{2,K} \otimes L_{2,K}$ into $L_{2,K}$ must use the coefficients. 

To make the claim precise we define 
\[	
  L_{2,R}^1 = \left\{ x \in L_{2,R} : x=\sum_{i=1}^n \alpha_i\beta_i^* \text{ for some distinct pairs } \alpha_i, \beta_i \in \{a,b\}^* \right\}.
\]
That is, $L_{2,R}^1$ consists of the elements of $L_{2,R}$ that can be written without coefficients. 

\begin{proposition} \label{prop: L2KtensorL2K field matters}
Let $K$ be a field and let $\phi \colon L_{2,K} \otimes L_{2,K} \to L_{2,K}$ be a unital $*$-homomorphism. 
If $K$ has characteristic $0$, then at least one of $\phi(a \otimes 1), \phi(b \otimes 1), \phi(1 \otimes a)$, and $\phi(1 \otimes b)$ will not be in $\newspan_{\Z}(L_{2,K}^1)$.
\end{proposition}
\begin{proof}
Because $K$ has characteristic $0$ we can, by the Cuntz-Krieger Uniqueness Theorem \cite[Theorem 6.5]{TomfordeLeavittOverRing}, embed $L_{2,\Z}$ into $\newspan_{\Z}(L_{2,K}^1)$ (note that $\newspan_{\Z}(L_{2,K}^1)$ is a $\Z$-algebra). 
As a $\Z$-module $L_{2,K}$ is torsion free, since $K$ has characteristic $0$, and hence it is flat as a $\Z$-module. 
Therefore we can use the observation from Remark \ref{rem: facts about flats!} to see that we also get an embedding of $L_{2,\Z} \otimes L_{2,\Z}$ into $L_{2,K} \otimes L_{2,K}$ by simply tensoring the embedding of $L_{2,\Z}$ into $\newspan_{\Z}(L_{2,K}^1)$ with itself. 

If all four of $\phi(a \otimes 1), \phi(b \otimes 1), \phi(1 \otimes a)$, and $\phi(1 \otimes b)$ are in $\newspan_{\Z}(L_{2,K}^1)$, then we can define a unital $*$-homomorphism $\psi \colon L_{2,\Z} \otimes L_{2,\Z} \to L_{2,\Z}$ such that 
\begin{eqnarray*}
  \psi(a \otimes 1) = \phi(a \otimes 1), & \psi(b \otimes 1) = \phi(b \otimes 1), \\
  \psi(1 \otimes a) = \phi(1 \otimes a), & \psi(1 \otimes b) = \phi(1 \otimes b).
\end{eqnarray*}
By construction $\psi$ is the restriction of $\phi$ when we think of $L_{2,\Z} \otimes L_{2,\Z}$ as sitting inside $L_{2,K} \otimes L_{2,K}$.
Since $L_{2,K} \otimes L_{2,K}$ is simple, $\phi$ is necessarily injective, and therefore $\psi$ is injective.
But that contradicts Theorem~\ref{thm: L2Z nonembedding}, so at least one of $\phi(a \otimes 1), \phi(b \otimes 1), \phi(1 \otimes a)$ and $\phi(1 \otimes b)$ will not be in $\newspan_{\Z}(L_{2,K}^1)$.
\end{proof}

We find this result to be interesting since it indicates that if $L_{2,K} \otimes L_{2,K}$ embeds into $L_{2,K}$ for some field of characteristic $0$, then the embedding must somehow reference the field. 

Our results also allows us to say something about potential embeddings for general rings. We say a unitary in an  involutive $R$-algebra has {\em full spectrum} if $q(\alpha) \neq 0$ for all nonzero $q \in R[x]$. (Here, the element $q(\alpha)$ means the obvious thing. See \cite[Remark~2.6]{BrownloweSorensenOne} for the details.)  

\begin{proposition}\label{prop: nonembedding of L2RtensorL2R}
Let $R$ be a commutative ring with unit.
Suppose that $\phi \colon L_{2,R}\otimes L_{2,R} \to L_{2,R}$ is an injective, unital $*$-homomorphism. 
If $u,v \in \cU(L_{2,R})$ have full spectrum then $\phi(u \otimes 1)$ and $\phi(1 \otimes v)$ cannot both be in $\cU_V$.
\end{proposition}

\begin{proof}
By Proposition \ref{prop: L2R embeds into tensor} and its proof we see that $z \mapsto u \otimes 1$ and $w \mapsto 1 \otimes v$ defines an injective $*$-homomorphism from $L_R[z,z^{-1},w,w^{-1}]$ to $L_{2,R} \otimes L_{2,R}$. 
Call it $\psi$. 
Then $\phi \circ \psi$ is an injective $*$-homomorphism from $L_R[z,z^{-1},w,w^{-1}]$ to $L_{2,R}$, and so by Theorem~\ref{thm: nonembedding of Laurent} we cannot have that both $\phi(\psi(z)) = \phi(u \otimes 1)$ and $\phi(\psi(w)) = \phi(1 \otimes v)$ are in $\cU_V$.
\end{proof}

Again the result gives an indication that the ring of coefficients is likely to matter to a given embedding. 
However, when working with rings that do not have characteristic $0$ we do not have that $\cU_V = \cU_1$ (Example \ref{ex:UV not U1}) and so exactly how the coefficients matter is murky.

\section{Note added in proof} \label{sec: added in proof}

After an initial version of this manuscript appeared on the arXiv, the second named author and Rune Johansen used the characterization of unitaries in $L_{2,\Z}$ given in Proposition \ref{prop: char of unitaries in L2Z} to give a similar characterization of the projections.
We can use this to conclude that not only is there no \emph{unital} $*$-homomorphic embedding of $L_{2,\Z} \otimes L_{2,\Z}$ into $L_{2,\Z}$ (Theorem~\ref{thm: L2Z nonembedding}) there is no $*$-homomorphic embedding.

\begin{proposition} \label{prop: all projections equivalent to the unit}
All nonzero projections in $L_{2,\Z}$ are Murray-von Neumann equivalent to the unit.
\end{proposition}
\begin{proof}
Let $p \in L_{2,\Z}$ be a projection.
By \cite[Theorem 4.5]{JohansenSorensen} there exist paths $\beta_1, \beta_2, \ldots, \beta_n \in \{a,b\}^*$ with disjoint cylinder sets such that 
\[
	p = \sum_{i=1}^n \beta_i \beta_i^*.
\]
Let $\alpha_1, \alpha_2, \ldots, \alpha_n$ be paths with disjoint cylinder sets whose union is $\{a,b\}^\N$. 
Then $t_i = \alpha_i \beta_i^*$ shows that $\alpha_i \sim \beta_i$, for $i=1,2,\ldots,n$. 
Since $\{\beta_i \beta_i^*\}$ and $\{\alpha_i \alpha_i^*\}$ are two collections of pairwise orthogonal projections, we then get (see \cite[Proposition 4.2.4]{BlackadarKTheoryBook}) that
\[
	p = \sum_{i=1}^n \beta_i \beta_i^* \sim \sum_{i=1}^n \alpha_i \alpha_i^* = 1. \qedhere
\]
\end{proof}

\begin{corollary}
There is no $*$-algebraic embedding of $L_{2,\Z}\otimes L_{2,\Z}$ into $L_{2,\Z}$. 
\end{corollary}
\begin{proof}
Let $\phi \colon L_{2,\Z}\otimes L_{2,\Z} \to L_{2,\Z}$ be a nonzero $*$-homomorphism and let $p = \phi(1 \otimes 1)$.
By Proposition~\ref{prop: all projections equivalent to the unit} we have $p \sim 1$, so we can pick a $t \in L_{2,\Z}$ such that $t^*t = 1$ and $tt^* = \phi(1 \otimes 1)$.
Then we can define a $*$-homomorphism $\psi \colon L_{2,\Z}\otimes L_{2,\Z} \to L_{2,\Z}$ by $\psi(x) = t^*\phi(x)t$. 
Note that $\psi$ is unital and that $\psi$ is an embedding if and only if $\phi$ is. 
But by Theorem~\ref{thm: L2Z nonembedding}, $\psi$ cannot be an embedding and hence $\phi$ is not an embedding. 
\end{proof}

\section*{Acknowledgements}

The authors thank Gene Abrams, Pere Ara, and Enrique Pardo for helpful comments on an earlier version of this paper that helped improve it. 
The second-named author is grateful for enlightening conversations about endomorphisms of $\OO_2$ and Thompson's group $V$ with Wojciech Szyma{\'n}ski. 

During work on this project the first named author visited the University of Oslo. 
The first named author would like to thank the Institute for Mathematics \& Its Applications at the University of Wollongong for providing financial support, and the University of Oslo and Nadia Larsen for also providing financial support, and for their hospitality. 

Part of this work was done while the second named author was supported by an individual post doctoral grant from the Danish Council for Independent Research \textbar {} Natural Sciences. 

\bibliographystyle{plain}

\begin{thebibliography}{10}

\bibitem{AbramsTheFirstDecade}
G.~Abrams.
\newblock Leavitt path algebras: the first decade.
\newblock {\em Bull. Math. Sci.}, 5(1):59--120, 2015.

\bibitem{AbramsAnhPardoMatrixRings}
G.~Abrams, P.N.~{\'A}nh, and, E.~Pardo. 
\newblock Isomorphisms between Leavitt algebras and their matrix rings.
\newblock {\em J. Reine Angew. Math.}, 624:103--132, 2008.

\bibitem{ALPS}
G.~Abrams, A.~Louly, E.~Pardo, and C.~Smith.
\newblock Flow invariants in the classification of Leavitt path algebras.
\newblock {\em J. Algebra}, 333(1):202--231, 2011.

%
\bibitem{AbramsAraSilesBook}
G.~Abrams, P.~Ara, and M.~Siles~Molina.
\newblock Leavitt path algebras.
\newblock In preperation.
%

\bibitem{AbramsPinoOriginalLeavitt}
G.~Abrams and G.~Aranda~Pino.
\newblock The {L}eavitt path algebra of a graph.
\newblock {\em J. Algebra}, 293(2):319--334, 2005.

%
%
%
\bibitem{AlahmadiAlsulamiLeavittDimension}
A.~Alahmadi, H.~Alsulami, S.~K. Jain, and E.~Zelmanov.
\newblock Leavitt path algebras of finite {G}elfand-{K}irillov dimension.
\newblock {\em J. Algebra Appl.}, 11(6):1250225, 6, 2012.
%
%

\bibitem{AraCortinasTensorProducts}
P.~Ara and G.~Corti{\~n}as.
\newblock Tensor products of {L}eavitt path algebras.
\newblock {\em Proc. Amer. Math. Soc.}, 141(8):2629--2639, 2013.

\bibitem{AraMorenoPardoKTheoryGraphAlgberas}
P.~Ara, M.~A. Moreno, and E.~Pardo.
\newblock Nonstable {$K$}-theory for graph algebras.
\newblock {\em Algebr. Represent. Theory}, 10(2):157--178, 2007.

%
%
\bibitem{BlackadarKTheoryBook}
B.~Blackadar.
\newblock {\em {$K$}-theory for operator algebras}, volume~5 of {\em
  Mathematical Sciences Research Institute Publications}.
\newblock Springer-Verlag, New York, 1986.


\bibitem{BlackadarOpAlgBook}
B.~Blackadar.
\newblock {\em Operator Algebras, Theory of 
$C^*$-Algebras and Von Neumann
Algebras}, Operator Algebras and Non-Commutative Geometry, vol. III, pp.
xx+517. Springer, Berlin, 2006.

%
\bibitem{BleakSalazarFreeProducts}
C.~Bleak and O.~Salazar-D{\'{\i}}az.
\newblock Free products in {R}. {T}hompson's group {$V$}.
\newblock {\em Trans. Amer. Math. Soc.}, 365(11):5967--5997, 2013.

\bibitem{BrownloweSorensenOne}
N. Brownlowe and A.~P.~W. S{\o}rensen.
\newblock Countable Leavitt $R$-algebras embed into $L_{2,R}$.

%
\bibitem{BourbakiCommutativeAlgebra}
N.~Bourbaki.
\newblock {\em Commutative algebra. {C}hapters 1--7}.
\newblock Elements of Mathematics (Berlin). Springer-Verlag, Berlin, 1989.
\newblock Translated from the French, Reprint of the 1972 edition.
%
\bibitem{BrinHigherDimensional}
M.~G. Brin.
\newblock Higher dimensional {T}hompson groups.
\newblock {\em Geom. Dedicata}, 108:163--192, 2004.

\bibitem{CannonFloydParry}
J. W. Cannon, W. J. Floyd and W. R. Parry, .
\newblock Introductory notes on Richard Thompson's groups.
\newblock {\em Enseign. Math.}, 42:215--256, 1996.


%
%
%
%
%
%
%
%
%
\bibitem{LeavittOriginal}
W.~G. Leavitt.
\newblock The module type of a ring.
\newblock {\em Trans. Amer. Math. Soc.}, 103:113--130, 1962.


\bibitem{JohansenSorensen}
R.~Johansen and A.~P.~W. S{\o}rensen.
\newblock The Cuntz splice does not preserve $*$-isomorphism of Leavitt path algebras over $\Z$.
\newblock arXiv:1507.01247, 2015.

\bibitem{NekrashevychCuntzPimsner}
V.~Nekrashevych.
\newblock Cuntz-{P}imsner algebras of group actions.
\newblock {\em J. Operator Theory}, 52(2):223--249, 2004.

\bibitem{PardoTheIsmomorphismProblem}
E.~Pardo.
\newblock The isomorphism problem for {H}igman-{T}hompson groups.
\newblock {\em J. Algebra}, 344:172--183, 2011.
%
%
%
%

\bibitem{RuizTomfordeInfiniteGraphs}
E.~Ruiz and M.~Tomforde.
\newblock Classification of unital simple Leavitt path algebras of infinite graphs.
\newblock {\em J. Algebra}, 384:45--83, 2013.

%
%
\bibitem{TomfordeLeavittOverRing}
M.~Tomforde.
\newblock Leavitt path algebras with coefficients in a commutative ring.
\newblock {\em J. Pure Appl. Algebra}, 215(4):471--484, 2011.

\end{thebibliography}


\end{document}